\documentclass[12pt,leqno]{amsart}

\usepackage{amsmath,amssymb,amsthm}
\usepackage{draftcopy}
\usepackage{hyperref}
\usepackage{graphics}
\usepackage{graphicx}
\usepackage{a4wide}

\usepackage{color}

\newtheorem{theorem}{Theorem}[section]
\newtheorem{proposition}[theorem]{Proposition}
\newtheorem{remark}[theorem]{Remark}
\newtheorem{lemma}[theorem]{Lemma}

\newtheorem{definition}[theorem]{Definition}

\DeclareMathOperator*{\limess}{lim\,ess}

\newcommand{\rr}{\mathbb{R}}
\newcommand{\sgn}{\,{\rm sgn}}

\def\eps{\epsilon}
\def\la{\lambda}

\def\l{\left}
\def\r{\right}
\def\umu{u^\mu}

\def\dd{\mathcal{D}}
\def\tumu{\tilde{u}_{\la}^{\mu}}

\begin{document}


\title[ Subcritical convection-diffusion equation with nonlocal diffusion]{On the asymptotic behavior of a subcritical convection-diffusion equation with nonlocal diffusion}

\author[C. M. Cazacu]{Cristian M. Cazacu}
\address[C. M. Cazacu]{Department of Mathematics and Informatics\\ Faculty of Applied Sciences\\
 University Politehnica of Bucharest\\
Splaiul Independentei 313\\ Bucharest\\ 060042\\ Romania
\hfill\break\indent \and
\hfill\break\indent
Institute of Mathematics ``Simion Stoilow'' of the Romanian Academy\\21 Calea Grivitei Street \\010702 Bucharest \\ Romania. Research group of the project PN-II-ID-PCE-2011-3-0075.}
\email{cristi\_cazacu2002@yahoo.com}

\author[L. I. Ignat]{Liviu I. Ignat}
\address[L. I. Ignat]{Institute of Mathematics ``Simion Stoilow'' of the Romanian Academy\\21 Calea Grivitei Street \\010702 Bucharest \\ Romania.}
\email{liviu.ignat@gmail.com}

\author[A. F. Pazoto]{Ademir F. Pazoto}
\address[A. F. Pazoto]{Instituto de Matem\'atica, Universidade Federal do Rio de
Janeiro, P.O. Box 68530, CEP  21941-909, Rio de Janeiro, RJ,  Brasil}
\email{ademir@im.ufrj.br}

\begin{abstract}
In this paper we consider a subcritical model that involves nonlocal diffusion and a classical convective term. In spite of the nonlocal diffusion, we obtain an Oleinik type estimate similar to the case when the diffusion is local. First we prove that the entropy solution can be obtained by adding a small viscous term $\mu u_{xx}$ and letting $\mu\to 0$. Then, by using uniform Oleinik estimates for the viscous approximation we are able to prove the well-posedness of the entropy solutions with $L^1$-initial data.
Using  a scaling argument and  hyperbolic estimates given by Oleinik's inequality, we  obtain the first term in the asymptotic behavior of the nonnegative solutions. Finally, the large time behavior of changing sign solutions is proved using the classical flux-entropy method and estimates for the nonlocal operator.

%

 \end{abstract}

\maketitle



\section{Introduction}

The aim of this paper is to extend the previous known results on the asymptotic behavior of classical convection-diffusion system \cite{MR1233647}
\begin{equation}\label{gas1}
\left\{
\begin{array}{ll}
u_t+|u|^{q-1}u_x=u_{xx},& x\in \rr,\ t>0,\\[10pt]
u(0)=\varphi
\end{array}
\right.
\end{equation}
to the following system
\begin{equation}\label{Ldifusion}
\left\{
\begin{array}{ll}
u_t+|u|^{q-1}u_x=Lu,& x\in \rr,\ t>0,\\[10pt]
u(0)=\varphi,
\end{array}
\right.
\end{equation}
where $1< q< 2$ and   $L$ is a nonlocal operator of the type
\begin{equation}\label{L}
(Lu)(x)=\int _{\rr}J(x-y)(u(y)-u(x))dy.
\end{equation}
Throughout this paper we assume that  the kernel $J$ satisfies the following general assumptions: $J\in L^1(\rr)$ is a nonnegative even  function with mass one. In view of these, when necessary for our purpose we will write $J\ast u-u$ instead of $Lu$. We recall that systems like \eqref{gas1}-\eqref{Ldifusion} for which the mass is conserved are so-called conservation laws.
System \eqref{Ldifusion} has been already studied for both the critical and supercritical case  $q\geq 2$ in the recent paper \cite{MR3190994} in which the  subcritical case $1<q< 2$ has been stated as an open problem. In the present work we are precisely interested in this case, in which the following nonlocal comparison principle plays a crucial role.

\begin{theorem}\label{main}
Let $L$ be an operator given by \eqref{L}. Then for any $\beta\geq 0$ and any nonnegative function $z\in L^\infty(\rr)$  there exists a nonpositive function
$A_z:\rr\rightarrow \rr$ such that
\begin{equation}\label{est.punctual}
\Big(zL(z^\beta w)-\frac{\beta}{\beta+1}wL(z^{\beta+1})\Big)(x_0)\leq  A_z(x_0) w(x_0)
\end{equation}
holds for any function $w\in C(\rr)\cap L^\infty(\rr)$ that
 attains its maximum at the point $x_0\in \rr$.
\end{theorem}

When we replace  $L$ by the classical Laplace operator the above inequality is satisfied for any $C^2(\rr)$ functions $w$ and $z$, with $z\geq 0$ (see Section \ref{positive}). This result was  used in \cite{MR1233647} (without being underlined)  to derive an Oleinik inequality for the solutions of system \eqref{gas1}.  However, it actually applies for more general operators, as the ones we study in the present paper.

It is by now classical for conservation laws that Oleinik estimates  represent a systematic tool to determine the asymptotic behavior of the corresponding solutions.   In general a mass conservation system has not a unique solution. Despite of this, the uniqueness is guaranteed when imposing an extra condition to be satisfied which is well-known as an entropy condition (see Definition \ref{def1}, which in particular applies to system \eqref{Ldifusion} with $f(u)=|u|^{q-1}u/q$ and $\alpha=1$).

As an application of Theorem \ref{main} we obtain Oleinik type estimates for the nonnegative solutions of \eqref{Ldifusion} as follows.

\begin{theorem}\label{oleinik_theorem}
Let $1< q\leq 2$ and $\varphi\in L^1(\rr)$ be a nonnegative initial datum with mass $M$. Then, the entropy solution $u$ of \eqref{Ldifusion} satisfies
\begin{equation}\label{ineq-oleinik}
\left(u^{q-1}\right)_x(t)\leq \frac{1}{t}, \quad \textrm{ in } \dd'(\rr), \quad \forall t>0.
\end{equation}
Moreover, there exists a constant $C(M, p, q)>0$ depending on $M, p$ and  $q$, such that
\begin{equation}\label{est.u.5}
\|u(t)\|_{L^p(\rr)}\leq C(M,p,q)t^{-\frac {1}q \left(1-\frac 1p\right)}, \quad \forall t>0, \  \forall 1\leq p\leq \infty.
\end{equation}
 \end{theorem}

The main result of this paper concerns the long time behaviour of the solutions of system \eqref{Ldifusion} and it is given in the following theorem.
\begin{theorem}\label{asimp} Assume that  $J\in L^1(\rr,1+|x|^2)$ and $1<q<2$.
For any  initial datum $\varphi\in L^1(\rr)$,  the entropy solution $u$ of system \eqref{Ldifusion} satisfies
\begin{equation}\label{lim.t}
\lim _{t\rightarrow \infty} t^{\frac  1q\left(1-\frac 1p\right)}\|u(t)-w_M(t)\|_{L^p(\rr)}=0, \quad 1\leq p < \infty,
\end{equation}
where $w_M$ is the unique entropy solution of the equation
 \begin{equation}\label{gas}
\left\{
\begin{array}{ll}
w_t  + ({|w|^{q-1}w}/q)_x= 0,& x\in \rr, \ t>0,\\[10pt]
w(x)=M\delta_0,
\end{array}
\right.
\end{equation}
  whenever the  mass $M$ of the initial datum is nontrivial.
  \end{theorem}

As proved in \cite[Section 2]{MR735207}, there exists a unique entropy solution $w_M$ (see Section \ref{entropyintro} for the precise meaning) to system \eqref{gas} which is given by the $N$-wave profile
\begin{equation}\label{wave_profile}
w_M(t,x)=\left\{
\begin{array}{ll}
(x/t)^{\frac 1{q-1}},  & 0<x<r(t),     \\
  0, & otherwise,
\end{array}
\right.
\end{equation}
with $r(t)=(\frac q{q-1})^{\frac{q-1}q}M^{(q-1)/q}t^{1/q}$.

Here, in contrast to the case $q \geq 2$,   the asymptotic behavior of the nonlocal system \eqref{Ldifusion} is given by  the convective part. We prove that, as time becomes large, the nonlocal diffusion term can be neglected. Actually, this phenomenon  also occurs for the local problem \eqref{gas1} as shown in \cite{MR1233647}. We predict that in the multidimensional space  $\rr^d$, $d\geq 2$,  similar results remain valid but in a new range of exponents which depends on the dimension, i.e. $1< q< (d+1)/d$. This comes in the spirit of the results obtained in \cite{MR1266100} and  \cite{MR1440033}, where the authors analyze the problem \eqref{gas1}. They proved that, in any dimension,  for very large time the effect of the diffusion is negligible as compared to convection precisely when $1< q< (d+1)/d$. Particularly, to overcome the difficulties which appeared when dealing with changing sign solutions the author in \cite{MR1440033} used  kinetic type arguments in order to prove the compactness of the rescaled solutions.
A  system similar to \eqref{gas1} was studied in \cite{MR2047911} when the nonlinearity $|u|^{q-1}u$ is changed to $|u|^q$. Whether the analysis presented here can be done for even nonlinearities remains to be investigated.

We will first prove
Theorem \ref{asimp} in the case of  nonnegative solutions by combining Oleinik type estimates with compactness arguments.
 When one tries to extend the results to changing sign solutions new difficulties appear. For the sake of completeness, we prefer to present them independently since different tools are used in their proofs. In particular, we want to emphasize the limitation of the Oleinik estimates which fail in the case of changing sign solutions. For the latter case we must apply more sophisticated arguments similar to Tartar \cite{MR584398} to prove compactness results.  This issue goes in the direction of the previous works
 \cite{1103.46310}, \cite[Th. 6.2, p. 128]{1214.45002}, \cite{MR2041005}, \cite{Ignat:2013fk},  where some compactness arguments were adopted to nonlocal problems. A flux-entropy method inspired in Tartar \cite{MR584398} but working when the convection is nonlocal as in \cite{MR2888353, MR2356418,Ignat:2013fk} is, as far as the authors know, an open problem.

%

The problem we address here can be also analyzed in the case of other type of operators in convolution form as, for example,  the  Fractional Laplacian:
$$(Lu)(x)=-(-\Delta)^{s} u(x)=c_{s}P.V. \int _{\rr}\frac{u(y)-u(x)}{|y-x|^{1+2s}}dy,$$
$0< s <1$,  or  Levy type operators
\[
Lu=P.V. \int _{\rr} (u(x+y)-u(x))K(y)dy
\]
 where $c_s$ is a universal constant depending only on $s$ and $K\in L^1\left(\rr,\frac{|x|^2}{1+|x|^2}\right)$ is a nonnegative kernel with  $K(x)=K(-x)$, respectively. In the particular case $K(x)=c_s |x|^{-1-2s}$ it corresponds to the fractional Laplacian $(-\Delta)^{s}$.
Moreover, these operators are infinitesimal generators of symmetry Levy processes. For more details on the convection-diffusion equations with nonlocal diffusion of Levy type we refer to \cite{MR1849690} and more recently to \cite{MR2914243}.
 Since both Levy-type and fractional Laplacian kernels do not belong to $L^1(\rr, 1+|x|^2)$ as required in Theorem \ref{asimp}, a similar analysis as the one we develop here cannot be carried out in the context of such operators. Despite of this  we emphasize that our estimate \eqref{est.punctual} given in Theorem \ref{main}  also holds when $L$ is of Levy type under suitable conditions on the functions $z$ and $w$.


The paper is organized as follows. In Section \ref{entropyintro} we discuss the notion of entropy solutions for systems \eqref{Ldifusion} and  \eqref{gas} and present some results obtained previously  in this context.  In Section \ref{positive} we prove Theorem \ref{main} and we apply it to obtain Oleinik type estimates for nonnegative solutions of some regularized system for  \eqref{Ldifusion}.
In Section \ref{secconv} we prove that the solutions of the regularized system converges to the entropy solution of system \eqref{Ldifusion}. As a consequence we prove Theorem \ref{oleinik_theorem}. Section \ref{asybehavior} contains the proof of Theorem \ref{asimp} in the case of nonnegative solutions. Finally, in Section \ref{signchange}
 we prove the asymptotic behaviour in Theorem \ref{asimp} in the case of changing sign solutions.

\section{Nonlocal conservation laws}\label{entropyintro}
Let us now recall some well-posedness results about conservation laws with nonlocal terms.

\begin{definition}\label{def1} Let $\alpha>0$, $f(u)$ be a locally Lipschitz  flux and $\varphi$ a bounded measurable function.
A measurable function $u\in L^\infty((0, \infty)\times \rr)$ is an entropy solution to
\begin{equation}\label{claws}
\left\{
\begin{array}{ll}
u_t+(f(u))_x=\alpha(J\ast u-u),& x\in \rr,\  t>0,\\[10pt]
u(0)=\varphi,
\end{array}
\right.
\end{equation}
 if satisfies the following conditions:

 C1) For every constant $k\in \rr$ the following inequality holds
\begin{align}\label{def}
\partial_t |u-k| +\partial _x &[\sgn(u-k)(f(u)-f(k))]\\
\nonumber&\leq
\alpha \sgn(u-k)J\ast (u-k)-\alpha |u-k|\quad \text{in}\ \mathcal{D}'((0,\infty)\times \rr),
\end{align}
i.e. for any $\phi\in C_c^\infty((0,\infty)\times \rr)$, $\phi\geq 0$,
\begin{multline*}
\int_0^\infty \int_{\rr} \left(|u-k|\frac{\partial \phi}{\partial t} +\sgn(u-k)(f(u)-f(k))\frac{\partial \phi}{\partial x}\right) dxdt\\
\geq \alpha  \int_0^\infty \int _{\rr}[|u-k|-\sgn(u-k)J\ast (u-k)]\phi dxdt,
\end{multline*}

C2) For every $R>0$
\begin{equation}\label{lim.zero}
\limess_{t\downarrow 0} \int_{|x|<R}|u(t,x)-u_0(x)|dx=0.
\end{equation}

\end{definition}
In particular, the election $\alpha=1$ and $f(u)=|u|^{q-1}u/q$, with $1< q\leq 2$, corresponds to an entropy solution of \eqref{Ldifusion}.

In the case of $BV$ initial data, the existence of a unique entropy solution of \eqref{claws} was proved  in \cite{0793.76005}, as well as,  the $L^1$-contraction property:
\begin{equation}\label{l1contract}
\|u_1(t)-u_2(t)\|_{L^1(\rr)}\leq \|\varphi_1-\varphi_2\|_{L^1(\rr)},
\end{equation}
where $u_i$ is the entropy solution of \eqref{claws} with the initial datum $\varphi_i\in BV$, $i\in \{1, 2\}$.
The existence is obtained by classical vanishing viscosity method: consider $\umu$ the solution of the regularizing problem

\begin{equation}\label{umu}
\left\{
\begin{array}{ll}
\umu_t+(f(\umu))_x=\alpha(J\ast \umu-\umu)+\mu \umu_{xx},& x\in \rr, \ t>0,\\[10pt]
\umu(0)=\varphi,
\end{array}
\right.
\end{equation}
and prove the compactness of the family $\{\umu\}_{\mu}$ (as $\mu$ tends to zero) in a suitable functional space. In order to make the reading easier, since $\alpha$ is fixed and $\mu$ is a parameter that tends to zero,  we omit the index $\alpha$ and only keep the notation $u^\mu$ when referring to the solution of systems like \eqref{umu}.

For $BV$ initial data $\varphi$, it was proved in \cite{0793.76005} that for any $t>0$, $\umu(t)$ converges in $L^1(\rr)$ to $u(t)$ the unique entropy solution of \eqref{claws}.
The case of  initial data $\varphi\in L^1(\rr)\cap L^\infty(\rr)$ was  considered in \cite{1052.35126}.  More precisely, the authors proved that for any $T>0$ the solution $u^\mu$ of \eqref{umu} satisfies
\[
\umu \rightarrow u \quad \text{in}\quad L^p_{loc}((0,T)\times \rr), \ 1\leq p<\infty, \quad \textrm{as } \mu\rightarrow 0,
\]
where the limit point $u\in C([0,T], L^1(\rr))\cap L^\infty((0,T)\times\rr)$ is an entropy solution of \eqref{claws}.
The $L^1(\rr)$-stability property \eqref{l1contract} also holds (cf. \cite[Theorem ~2.5]{1052.35126}) in this case without requiring $BV$ regularity of the  solutions.

Regarding the uniqueness of the entropy solution of system \eqref{claws}, i.e. functions that satisfy conditions C1) and C2), in \cite{1052.35126} it was proved that for any $\varphi \in L^1(\rr)\cap L^\infty(\rr)$ and any $T> 0$ there exists a unique solution in the class $L^\infty((0,T);L^1(\rr)\cap L^\infty(\rr))$.
The case when the initial datum belongs to $L^\infty(\rr)$ was considered in \cite[Theorem 2, p.~497]{MR2103702} under the additional assumption that $J\in L^1(\rr,1+|x|)$. Finally the uniqueness in the class $L^\infty((0,\infty),L^1(\rr))\cap L^\infty_{loc}((0,\infty),L^\infty(\rr))$ was proved in \cite[Section~1.4, p.~493]{MR2103702}.
%

In the case when $f(u)=|u|^{q-1}u/q$, $1<q\leq 2$ we complete the results in \cite{1052.35126} considering the case of $L^1(\rr)$ initial data.


  \begin{theorem}\label{global.convergence}
 Let  $\alpha>0$, $1<q\leq 2$ and $f(u)=|u|^{q-1}u/q$. For any  $\varphi\in L^1(\rr)$
  	there exists a unique entropy solution of system \eqref{claws} in the class
\[
  u\in L^\infty([0,\infty),L^1(\rr))\cap L^\infty_{loc}((0,\infty),L^\infty(\rr)).
\]
  	Moreover,   	the solutions $u^\mu$ of system \eqref{umu} satisfy
  	\begin{equation}
\label{conv.fuerte}
  \umu\rightarrow u\quad \text{in}\quad C([0,T],L^1(\rr)),
\end{equation}
and
\[
  \umu(t)\rightarrow u(t)\quad \text{in}\quad L^p(\rr), \quad  \forall t>0,\ \forall 1\leq p<\infty,
\]
as $\mu\to 0$.
  \end{theorem}
The proof of Theorem \ref{global.convergence} is given at the end of Section \ref{secconv}.

%
%

%
%
%
%
%

We recall that the case
 $q=2$ was considered in \cite{1070.35065} when the initial datum belongs to $L^1(\rr)\cap L^\infty(\rr)$. However, the results obtained in \cite{1070.35065} allow to prove the well-posedness requiring only $L^1(\rr)$-initial data.

Let us now say  few words about the case $\alpha=0$ and $f(u)=|u|^{q-1}u/q, \ q>1.$

\begin{definition}\label{entropysol} By an entropy solution of system \eqref{gas}  we mean a function
\[
w\in L^\infty((0,\infty),L^1(\rr))\cap L^\infty((\tau,\infty)\times \rr), \ \forall \tau\in (0,\infty)
\]
such that:

 C3) For every constant $k\in \rr$ and $\phi\in C_c^\infty((0,\infty)\times \rr)$, $\phi\geq 0$,  the following inequality holds
\begin{align*}
\int_0^\infty \int_{\rr} \left(|w-k|\frac{\partial \phi}{\partial t} +\sgn(w-k)(f(w)-f(k))\frac{\partial \phi}{\partial x}\right) dxdt\geq 0,
\end{align*}

C4) For any bounded continuous function $\psi$
\begin{equation}\label{lim.zero.2}
\limess_{t\downarrow 0} \int _\rr w(t,x)\psi(x)dx=M\psi(0).
\end{equation}
\end{definition}

The existence of an unique entropy solution of system \eqref{gas}, as well as its properties were deeply analyzed in \cite{MR735207}. For further details we suggest the reader to explore the quoted paper \cite{MR735207}.

\section{Oleinik type estimates}\label{positive}

 In this section our first goal is to prove Theorem \ref{main}.  Secondly, we  apply Theorem \ref{main} to obtain Oleinik estimates and uniform $L^p$-bounds with respect to $\alpha$ and $\mu$ for the solutions of the regularized system \eqref{umu}. In particular, we prove that estimates \eqref{ineq-oleinik} and   \eqref{est.u.5} in Theorem \ref{oleinik_theorem} are verified for  the regularized solutions of \eqref{umu} (see e.g. Lemma \ref{linfty}).

\subsection{Comparison principle for the nonlocal operator}\label{nonlocalMP}
As a motivation, we emphasize that in the case of the classical Laplace operator, $Lu=u_{xx}$, explicit computations  yield to:
\begin{align*}
zL(z^\beta w)-\frac{\beta}{\beta+1}wL(z^{\beta+1})&=
z\left(\beta(\beta-1)z^{\beta-2}z_x^2 w+\beta z^{\beta-1}z_{xx}w+2\beta z^{\beta-1}z_xw_x +z^\beta w_{xx}\right) \\
&\quad \quad -\beta w\left(\beta z^{\beta-1}z_x^2 +z^\beta z_{xx}\right) \\
&= z^{\beta+1}w_{xx}+2\beta z^{\beta}z_x w_x- \beta z^{\beta-1}z_x^2 w.
\end{align*}
Observe that assuming $x_0$ is the point where $w$ attains its maximum ($w(x_0)=\max _{\rr} w$) we have
$$\l(zL\l(z^\beta w\r)-\frac{\beta}{\beta+1}wL\l(z^{\beta+1}\r)\r)(x_0)\leq -\beta z^{\beta-1}z_x^2w(x_0).$$
This shows that in this case we can choose $A_z(x)=-\beta z^{\beta-1} z_x^2 (x)$ and the estimate of Theorem \ref{main} holds.

\begin{proof}[Proof of Theorem \ref{main}]
Let us now consider an integral operator in the form \eqref{L}. It follows that
\begin{align*}
I_\beta(x)&=\left(zL(z^\beta w)-\frac{\beta}{\beta+1}wL(z^{\beta+1})\right)(x)\\
&=z(x)\int _{\rr}J(x-y)\l(z^\beta(y) w(y)-z^\beta(x) w(x)\r)dy\\
&-\frac \beta{\beta+1} w(x)\int _{\rr}J(x-y)\l(z^{\beta+1}(y)-z^{\beta+1}(x)\r)dy.
\end{align*}
Let $x_0$ be the point where $w$ attains its maximum. Since $J$ has mass one, $J\geq 0$  and $z$ is nonnegative we can write
\begin{align*}
I_\beta(x_0)&=\int _{\rr}J(x_0-y)\left( z(x_0)z^\beta(y) w(y)-\frac{\beta}{\beta+1} w(x_0)z^{\beta+1}(y)-\frac 1{\beta+1} w(x_0)z^{\beta+1}(x_0)\right)dy\\
&\leq w(x_0)\int _{\rr}J(x_0-y)\left( z(x_0)z^\beta(y) -\frac{\beta}{\beta+1} z^{\beta+1}(y)-\frac 1{\beta+1} z^{\beta+1}(x_0)\right)dy.
\end{align*}
Denoting by
$$A_z(x):= \int _{\rr}J(x-y)\left( z(x)z^\beta(y) -\frac{\beta}{\beta+1} z^{\beta+1}(y)-\frac 1{\beta+1} z^{\beta+1}(x)\right)dy,$$ by  Young's inequality we obtain that $A_z(x)\leq 0$ and the desired inequality holds.
\end{proof}

\subsection{Regularized system}\label{reg}

%


For $1<q< 2$ let us now consider the regularized problem
\begin{equation}\label{umu.1}
\left\{
\begin{array}{ll}
\umu_t+|\umu|^{q-1} (\umu)_x=\alpha(J\ast \umu-\umu) +\mu \umu_{xx},& x\in \rr,\ t>0,\\[10pt]
\umu(0)=\varphi.
\end{array}
\right.
\end{equation}
 We emphasize that all the results obtained in this section are uniform with respect to the positive parameters $\alpha$ and $\mu$.

Following closely the analysis done in \cite{0762.35011},
for any $\varphi\in L^1(\rr)$  we obtain a unique solution $\umu \in C([0,\infty),L^1(\rr))$ of \eqref{umu.1} that satisfies
\[
\umu \in  C((0,\infty),W^{2,p}(\rr))\cap C^1((0,\infty),L^p(\rr)), \ 1<p<\infty.
\]
When integrating equation \eqref{umu.1} in the space variable we obtain that
\begin{equation}\label{conservation.property}
  \int _{\rr}\umu(t,x)dx=\int _{\rr} \varphi(x)dx.
\end{equation}
 These results, analyzed in details in \cite{0762.35011} by using the smoothing properties of the heat kernel, are better than the ones in \cite[Appendix A]{1052.35126}  which deal only with the case of $L^1(\rr)\cap L^\infty(\rr)$ initial data. Moreover, the solution is nonnegative if the initial data is nonnegative as a consequence of the following comparison principle.

\begin{proposition}[Comparison Principle]\label{max.principle}
Assume $\varphi, \tilde{\varphi}\in L^1(\rr)$ and let $\umu, \tilde{u}^{\mu}$ be the corresponding  solutions of \eqref{umu.1} with the initial data $\varphi$ and $\tilde{\varphi}$ respectively. Then it holds,
\begin{equation}\label{mp1}
 \quad \int_{\rr} (\tilde{u}^\mu(t, x)-u^\mu(t, x))^+ dx \leq \int_{\rr} (\tilde{\varphi}(x)-\varphi(x))^{+} dx, \quad \forall t>0.
 \end{equation}
 In addition,
 \begin{equation}\label{mp2}
 \|\tilde{u}^{\mu} (t, \cdot)-\umu(t, \cdot)\|_{L^1(\rr)}\leq \|\tilde{\varphi}-\varphi\|_{L^1(\rr)},
 \quad \forall t>0.
 \end{equation}
\end{proposition}
\begin{remark}\label{rem1}
An immediate consequence is the following comparison principle: if $\varphi, \tilde{\varphi}$ are two initial data such that $\varphi\leq \tilde{\varphi}$ then the corresponding solutions $u^
\mu, \tilde{u}^\mu$ satisfy $u^
\mu \leq \tilde{u}^\mu$.
\end{remark}
\begin{proof}We write the equation satisfied by $u^\mu - \tilde{u}^\mu $ and multiply it by $\sgn(u^\mu - \tilde{u}^\mu)^+$.
	Integrating the result we obtain the desired estimate since for any function $v\in L^1(\rr)$ we have
	\[
  \int_\rr (J \ast v-v)\sgn(v^+)dx\leq 0.
\]
A similar argument works for the second estimate.
\end{proof}

We complete the  results presented in Section \ref{nonlocalMP} with an Oleinik type estimate for the solutions of \eqref{umu.1}, estimate that will help us to prove later one of the main results of this paper,  Theorem \ref{asimp}. We point out that when $1<q<2$, $\alpha=0$ and $\mu=1$ the result was already obtained  in \cite{MR1233647}.
The same analysis in the case $q=2$ was obtained in
\cite{1070.35065}.

\begin{theorem}\label{oleinik} Let us consider two positive parameters $\alpha$ and $\mu$  and let $\varphi\in L^1(\rr)$ be nonnegative. Then, the  solution of system \eqref{umu.1} satisfies
\begin{equation}\label{one-side}
((\umu)^{q-1})_x(t,x)\leq \frac 1t, \quad    \forall t>0,  \textrm{ a.e. } x\in \rr.
\end{equation}
\end{theorem}

 \begin{proof}
Let us start as in \cite{MR1233647}. By approximating our solution by uniformly positive and bounded solutions we can  consider $\eps>0$ and an initial data $\varphi \in C^\infty(\rr)$ such that $\epsilon\leq \varphi \leq M$. This implies that $u^\mu$ is a uniformly bounded positive classical solution. Moreover, $\umu\in C^k_b([0,T]\times \rr)$, $k\geq 1$, with a norm that depends on $T$ and on the $C^k(\rr)$-norm of $\varphi$. 
We prove estimate \eqref{one-side} for such solutions and then passing to the limit as $\eps$ tends to zero the result in \eqref{one-side} holds for any nonnegative initial data $\varphi \in L^1(\rr)$.

Let us set $z=(\umu)^{q-1}$. 
For simplicity we avoid to emphasize the dependance of $z$ on $\mu$. Properties of $\umu$ transfer to $z$, so $z\in C^4_b([0,T]\times \rr)$.
Then we have
\[
u_t^\mu=\frac 1{q-1}z^{\frac 1{q-1}-1}z_t, \quad u_x^\mu=\frac 1{q-1}z^{\frac 1{q-1}-1}z_x,
\]
\[
u_{xx}^\mu=\frac 1{q-1}\l(\frac 1{q-1}-1\r)z^{\frac 1{q-1}-2}z_x^2+\frac 1{q-1}z^{\frac 1{q-1}-1}z_{xx}.
\]
It follows that $z$ verifies the equation
\[
z_t=(q-1)z^{1-\frac 1{q-1}} L\l(z^{\frac 1{q-1}}\r)-zz_x+\mu \l(\beta \frac{z_x^2}z+z_{xx}\r),
\]
where $\beta=\frac{2-q}{q-1}\geq 0$.
Putting $w=z_x$ and using that $L=\alpha(J\ast u-u)$  it follows that $w\in C_b^3([0,T]\times \rr)$ is a solution of
\[
w_t+zw_x+w^2+\mu
\l( \beta\frac{w^3}{z^2}-2\beta \frac{w}z w_x-w_{xx}
\r) =z^{\frac{q-2}{q-1}} L\l(z^{\frac{2-q}{q-1}}w\r)- (2-q)z^{-\frac 1{q-1}}w L\l(z^{\frac 1{q-1}}\r).
\]
Consequently,
\[
w_t+zw_x+w^2+\mu
\l( \beta\frac{w^3}{z^2}-2\beta \frac{w}z w_x-w_{xx}
\r) =z^{-1-\beta} \left( zL \l(z^\beta w\r)-\frac {\beta}{\beta+1} w L\l(z^{\beta+1}\r) \right).
\]

Let us denote $W(t)=\sup _{x\in \rr} w(t,x)$. Since $u^\mu$ is a uniformly bounded classical solution we can apply the same arguments as in \cite[Th. 1.18]{MR2962829} to show that $W$ is locally Lipschitz. In particular $W$ is absolutely continuous so differentiable almost everywhere and for any $0<t_1<t_2<\infty$ it satisfies
\[
  W(t_2)-W(t_1)=\int _{t_1}^{t_2} W'(s)ds.
\]

We now differentiate $W(t)$ for $t>0$ and obtain the equation satisfied by it. Let us choose $0< s<t$ and use the Taylor expansion in the time variable $t$:
\begin{align}
\nonumber w(x, t)&\leq w(x, t-s)+s w_t(x, t)+ C s^2 \\
\label{ineg.w}&\leq W(t-s)+s w_t (x, t) + C s^2.
\end{align}

In the case when for each fixed $t\geq 0$, function $w(t,x)$ attains its maximum at a point $x_t\in \rr$ we easily obtain by using 
  Theorem \ref{main}  that, for some function $A\leq 0$,    $W$ satisfies the following inequality
\[
W'(t)+W^2(t)+\mu \beta\frac{W^3(t)}{z^2(t)}\leq \frac{W(t)}{z^{\beta+1}(t)} A(t), \quad a.e.\ t>0.
\] 

It may happen that for some positive times $t$ the maximum of $w(t,x)$ is attained at $x=\infty$. In this case we slightly modify the arguments above to obtain a differential inequality for $W$.
Let us now consider the point $x=x_n$ such that $w(x_n , t)=W(t)-1/n$. We recall the following properties of sequence $\{w(t,x_n)\}_{n\geq 0}$ proved in \cite[Lemma~1.17]{MR2962829}  
\begin{equation}
\label{lim.prop.w}
  \lim _{n\rightarrow\infty} w_x(t,x_n)\rightarrow 0, \quad  \limsup _{n\rightarrow\infty}w_{xx}(t,x_n)\leq 0.
\end{equation}

Now we evaluate \eqref{ineg.w} at the point $x=x_n$.
\begin{align}\label{eva}
w(t,x_n,)\leq W(t-s) +s& \Bigg[- z(t,x_n ) w_x(t,x_n ,)-w^2 (t,x_n)\nonumber\\
&\quad -\mu\Big(\beta \frac{w^3(t,x_n)}{z^2(t,x_n)}-2\beta \frac{w(t,x_n)}{z(t,x_n)} w_x(t,x_n) -w_{xx}(t,x_n) \Big)\nonumber\\
&\quad+ z^{-1-\beta}(t,x_n)\Big(z L(z^\beta w)-\frac{\beta}{\beta+1} wL(z^{\beta+1})\Big)(t,x_n)
 \Bigg].
\end{align}
Recall that $\{z(x_n, t)\}_n$ is uniformly bounded. So, we can extract a subsequence such that $z(x_n, t)\rightarrow p(t)$, as $n\rightarrow \infty$, where
$\eps^{1/(q-1)}\leq p(t)\leq M^{1/(q-1)}$.

Let us analyze the last term in \eqref{eva}. At the point $x=x_n$ we have
\begin{align*}
I(t):&=\Big(z L(z^\beta w)-\frac{\beta}{\beta+1} w L(z^{\beta+1})\Big)(t,x_n)\\
&=\int_{\rr} J(x_n -y)\left(z(x_n) z^\beta(y) w(y)-\frac{\beta}{\beta+1} w(x_n) z^{\beta+1}(y) -\frac{1}{\beta+1} w(x_n) z^{\beta+1}(x_n)  \right) dy.
\end{align*}
We use that $w(t,x_n)=W(t)-1/n$. Thus
\begin{align}\label{baba}
I(t)&\leq W(t)\int_{\rr}J(x_n-y) \left(z(x_n) z^\beta(y) -\frac{\beta}{\beta+1} z^{\beta+1}(y)-\frac{1}{\beta+1} z^{\beta+1}(x_n) \right) dy\nonumber\\
& +\frac{1}{n}\int_{\rr}J(x-y) \left( \frac{\beta}{\beta+1} z^{\beta+1}(y) +\frac{1}{\beta+1} z^{\beta+1} (x_n)\right)dy\nonumber\\
&\leq W(t) A_{z(t)}(x_n )+ \frac{M^{\frac{\beta+1}{q-1}}}{n},
\end{align}
where 
\[
  A_z(x):=\int_{\rr}J(x-y) \left(z(x) z^\beta(y) -\frac{\beta}{\beta+1} z^{\beta+1}(y)-\frac{1}{\beta+1} z^{\beta+1}(x) \right) dy\leq 0
  \quad \forall\, x\in \rr.
\]
Since $|A_{z(t)}(x)|\leq 2 M^{\frac{\beta+1}{q-1}}$ we have, up to a subsequence, that
$$A_{z(t)}(x_n ) \rightarrow A(t)\leq 0, \quad n\rightarrow \infty.$$
Passing to the limit, for a subsequence of $\{x_n\}_{n}$ and also applying \eqref{lim.prop.w} we have
\begin{align*}
W(t)&\leq W(t-s) + s\left(-W^2(t)-\mu \beta \frac{W^3(t)}{p^2(t)}+ p^{-1-\beta}(t) W(t) A(t)  \right)
\end{align*}
 Letting $s\rightarrow 0$ we have
 $$W'(t)\leq -W^2(t)-\mu \beta \frac{W^3(t)}{p^2(t)}+\frac{W(t)A(t)}{p^{1+\beta}(t)},$$
 at any point $t$ where $W$ is differentiable.
 Multiplying the above equation by $\sgn(W^+)$, where $W^+$ is the positive part of $W$,  we obtain that  $W^+$ satisfies
\[
(W^+)'(t)+(W^+)^2(t)\leq 0, \quad a.e.\ t>0.
\]
Denoting $\eta(t)=W(t)-\frac 1t$ we obtain that $\eta'(t)+2\eta(t)/t+\eta^2(t)\leq 0$ for a.e. $t>0$. So $\eta^+ $ satisfies $(\eta^+)'\leq 0$  for a.e. $t>0$. Since $\eta^+$ is also an  absolutely continuous function we obtain that $\eta^+$ is a nonincreasing function: for any  $0<s<t$ we have $\eta^+(t)-\eta^+(s)=\int _s^t (\eta^+)'(\sigma)d\sigma\leq 0$. Moreover $\lim _{s\rightarrow 0} \eta^+(s)=0$ so $\eta^+(t)\leq 0$ for all $t>0$ so $W^+(t)\leq 1/t$ for all $t>0$:
We thus obtained exactly the  desired estimate for smooth solutions.
\[
((u^\mu(t, x))^{q-1})_x\leq \frac 1t.
\]
The proof is now complete.
\end{proof}

We emphasize that all the above results are independent of the parameters $\alpha$ and
$\mu$ in \eqref{umu.1}. As a consequence of them  we can deduce some extra properties of the solutions of equation \eqref{umu.1}.

\begin{lemma}\label{linfty}
For any $\varphi\in L^1(\rr)$ nonnegative, the solution $\umu$ of system \eqref{umu.1} satisfies
\begin{equation}\label{est.infty}
0\leq \umu (t,x)\leq \Big(\frac {qM}{(q-1)t}\Big)^{1/q}, \quad \forall t>0, \textrm{ a.e. } x
\in \rr,
\end{equation}


\begin{equation}\label{est.5}
\|\umu(t)\|_{L^p(\rr)}\leq C(M,p, q)t^{-\frac {1}q (1-\frac 1p)}, \quad \forall\  1\leq p< \infty, \quad \forall t>0,
\end{equation}
\begin{equation}\label{estt.6}
 \umu_x(t, x) \leq  C(M, q) t^{-2/q}, \quad \forall t>0, \textrm{ a.e. } x\in \rr,
\end{equation}
uniformly with respect to the positive parameters $\mu$ and $\alpha$, where $C(M, q)$ and $C(M, p, q)$ are positive constants. We point out that these constants are independent of $J$.
\end{lemma}
\begin{proof}
Estimates \eqref{est.infty} and \eqref{est.5} follow as in \cite{MR1233647}, so we omit to prove them. In fact the lower bound in \eqref{est.infty} emanates from Remark \ref{rem1}.   Estimate \eqref{estt.6} follows from \eqref{one-side} by using that $1< q<2$ and estimate \eqref{est.infty}:
$$\umu_{ x}(t, x)\leq \frac{(\umu)^{2-q}(t,x)}{(q-1)t}\leq
C(M, q)t^{-2/q}, \quad \forall t>0,  \textrm{ a.e. } \ x\in \rr.$$
The proof is now finished.
\end{proof}

%
%

\section{Convergence to the nonregularized system}\label{secconv}

 The main goal of this section is to prove some compactness results
 for the rescaled solutions depending on a new parameter $\lambda$ that we will introduce below.
By means of these compactness results, through a convergence argument with respect to the parameters $\lambda$ and $\mu$,  we will be able to analyze in more details  the initial system
\begin{equation}\label{claws.1}
\left\{
\begin{array}{ll}
u_t+|u|^{q-1}u_x=J\ast u-u,& x\in \rr,t>0,\\[10pt]
u(0)=\varphi.
\end{array}
\right.
\end{equation}

First, since we are interested in the asymptotic behavior of the solutions of   system \eqref{claws.1} we introduce the rescaled solutions $u_\lambda(t,x)=\lambda u(\lambda^q t, \lambda x)$, with $\lambda>0$. The new family satisfies the system
\begin{equation}\label{claws.2}
\left\{
\begin{array}{ll}
u_{\la,t}+|u_\la|^{q-1}(u_\la)_x=\la^{q} (J_\la\ast u_\la -u_\la),\, & x\in \rr,\ t>0,\\[10pt]
u_\la(0)=\varphi_\la,
\end{array}
\right.
\end{equation}
where $J_\la$ is defined by $J_\la(x)=\la J(\la x)$ and $\varphi_\la(x)=\la \varphi(\la x)$.

Observe that this is a hyperbolic scaling in contrast with the parabolic one  used in \cite{MR3190994, Ignat:2013fk}. This is due to the fact that we analyze the sublinear case $1<q<2$ where the convection is dominant.

The family $\{u_\la\}$ will play a key role in Section \ref{asybehavior} when analyzing the long time behavior of the solution of system \eqref{claws.1}. In order to obtain  estimates for $u_\la $ we introduce a viscous approximation $\umu_\la$ of system \eqref{claws.2} by adding a $\mu$-viscous term. Then we  prove that the approximate solution $u_\la^\mu$ converges to $u_\la$ as $\mu$ tends to zero,  and then transfer all estimates verified for $\umu_\la$ to $u_\la$.
In the particular case $\lambda=1$ we extend in Theorem \ref{global.convergence} the results of
\cite[Theorem 2.4]{1052.35126}. The proof of Theorem \ref{global.convergence} will be given at the end of this section.





For each $\mu, \la>0$ we consider $\umu_\la$ the solution of the following system
\begin{equation}\label{eq.u.la}
\left\{
\begin{array}{ll}
\umu_{\la,t}+|\umu_\la|^{q-1}(\umu_\la)_x=\la^{q} (J_\la\ast \umu_\la -\umu_\la) +\mu (\umu_\la)_{xx},& x\in \rr,\ t>0,\\[10pt]
\umu_\la(0)=\varphi_\la.
\end{array}
\right.
\end{equation}

All the estimates obtained in Section \ref{reg} hold for $\umu_\lambda $  uniformly on $\lambda$ and $\mu$ since $J_\lambda$ has mass one and the results in Lemma \ref{linfty} do not depend on the parameter $\alpha$ in \eqref{umu.1}. The mass conservation \eqref{conservation.property} also holds.

We now transfer estimates \eqref{est.infty}-\eqref{est.5} obtained in Lemma \ref{linfty} to changing sign solutions of system \eqref{eq.u.la} by comparing them with some nonnegative solutions for which such estimates hold true.

\begin{lemma}
\label{est.sol.negativas}
	For any $\varphi\in  L^1(\rr)$ the solution $\umu_\la$ of equation \eqref{eq.u.la} satisfies
	\begin{equation}
\label{est.neg.1}
   |\umu_\lambda (t,x)|\leq \Big(\frac {q\|\varphi\|_{L^1(\rr)}}{(q-1)t}\Big)^{1/q}, \quad \forall t>0, \quad \textrm{ a.e. } x\in \rr,
\end{equation}

\begin{equation}\label{est.5.2}
\|\umu_\la(t)\|_{L^p(\rr)}\leq C(\|\varphi\|_{L^1(\rr)},p, q)t^{-\frac {1}q (1-\frac 1p)}, \quad \forall 1\leq p\leq \infty, \quad \forall t>0.
\end{equation}

\end{lemma}

\begin{proof}Let us consider $\tilde{u}_\lambda^\mu$ to be the solution of \eqref{eq.u.la} with the nonnegative initial data $\tilde \varphi_\la=|\varphi_\la|$.
Using the comparison principle in  Proposition \ref{max.principle} (and more precisely Remark \ref{rem1}) we get
\begin{equation}\label{comp_principle}
|\umu_\la|\leq \tumu.
\end{equation}
Since $\tumu$ is nonnegative the estimates in Lemma \ref{linfty}
hold for $\tumu$.	 Hence estimates \eqref{est.infty} and \eqref{est.5}  transfer to the function $\umu_\lambda$. These complete the proof.
\end{proof}

The following results will be used when considering the long time behavior of changing sign solutions.
\begin{lemma}
\label{est.gradient}
For any  $\varphi\in  L^1(\rr)$ and  $0<t_1<t_2<\infty$ we have
\begin{equation}\label{gradient.est}
  \lambda^q\int_{t_1}^{t_2}\int_\rr \int _\rr J_\lambda(x-y)(\umu_\lambda(t,x)-\umu_\lambda(t,y))^2dxdydt\leq C(t_1,\|\varphi\|_{L^1(\rr)}).
\end{equation}
\end{lemma}
\begin{proof}
	We multiply equation \eqref{eq.u.la} by $\umu_\lambda$ and integrate on $(t_1,t_2)\times \rr$. We obtain
	\begin{align*}
  \|\umu_\lambda(t_2)\|_{L^2(\rr)}^2&+{\lambda^q }\int_{t_1}^{t_2}\int_\rr \int _\rr J_\lambda(x-y)(\umu_\lambda(t,x)-\umu_\lambda(t,y))^2dxdydt +2\mu\int_{t_1}^{t_2}\int_{\rr } (\umu_{\lambda,x})^2dxdt \\
  &=  \|\umu_\lambda(t_1)\|_{L^2(\rr)}^2.
\end{align*}
Using the decay of the $L^2$-norm obtained in  Lemma \ref{est.sol.negativas} we obtain the desired result.
\end{proof}

We now uniformly control the tails of the solutions to system \eqref{eq.u.la}.

\begin{lemma}\label{tails.control} Let us assume that $J\in L^1(\rr, 1+|x|^2)$ and let $\varphi\in L^1(\rr)$. For any parameters $\lambda\geq 1$ and $0<\mu\leq 1$, there exists a positive constant $C=C(\|\varphi\|_{L^1(\rr)}, J, q)$ such that the solution of system \eqref{eq.u.la} satisfies the following estimate
\begin{equation}\label{int.2r}
\int _{|x|>2R} |\umu_\la (t,x)|dx\leq \int _{|x|>R}|\varphi(x)|dx+C\left(\frac{t}{R^2}+\frac {t^{1/q}}{R}\right)
\end{equation}
 for any $t>0$,  $R>0$.
\end{lemma}

\begin{proof} Applying the comparison principle in Proposition  \ref{max.principle} we remark that it is sufficient to consider the case of nonnegative solutions. Indeed, choosing as initial data $\tilde \varphi_\la =|\varphi_\la|$ we obtain that $|\umu_{\la}(t,x)|\leq \tilde{u}_{\la}^{\mu}(t,x)$. We then use estimate \eqref{int.2r} for $\tilde{u}_{\la}^{\mu}$ and the proof finishes.

We now prove \eqref{int.2r} for nonnegative solutions $\tilde{u}_{\la}^{\mu}$.
Let us choose  a nonnegative test function $\psi\in C^2(\rr)$.
 We multiply equation \eqref{eq.u.la} by $\psi$ and integrate by parts:
\begin{multline}\label{est.6}
\int_{\rr} \tilde{u}_{\la}^{\mu}(t,x) \psi(x)dx-\int _{\rr}|\varphi_\la(x)|\psi(x)dx\\
=\la^q \int_0^t \int_{\rr} \tilde{u}_{\la}^{\mu}(s, x)(J_\la \ast \psi-\psi)(x) dxds +\int_0^t \int_{\rr} \frac{(\tilde{u}_{\la}^{\mu}(s, x))^q}{q}\psi_x dxds\\
+ \mu \int_0^t \int_{\rr} \tilde{u}_{\la}^{\mu}(s, x) \psi_{xx}(x) dx ds.
\end{multline}
We now recall that (see for example  \cite[Lemma 2.2]{Ignat:2013fk}) for any  $J\in L^1(1+|x|^2)$ with mass one and any $p\in [1, \infty]$
there exists $C(J,p)>0$  such that the following inequality holds for any $\lambda > 0$:
\begin{equation}
\label{bound.second.order}
   \|\la^2 (J_\la \ast \psi - \psi)\|_{L^p(\rr)} \leq C(J,p) \|\psi_{xx}\|_{L^p(\rr)}.
\end{equation}
 This and the mass conservation property \eqref{conservation.property} imply that
 \begin{align*}
  \la^q \int_0^t \int_{\rr} |\tilde{u}_{\la}^{\mu}(s, x)(J_\la \ast \psi-\psi)(x)| dxds
  & \leq \lambda^{q-2}C(J) t \|\psi_{xx}\|_{L^\infty(\rr)}\|\varphi_\la\|_{L^1(\rr)}\\
  & = \lambda^{q-2}C(J) t \|\psi_{xx}\|_{L^\infty(\rr)}\|\varphi\|_{L^1(\rr)}.
\end{align*}

On the other hand, using estimate \eqref{est.5.2} for the solution of system \eqref{eq.u.la} we  successively have
\begin{align*}
\int_{0}^{t} \int_{\rr}|(\tilde{u}_{\la}^{\mu}(s, x))^q \psi_x | dx ds
&\leq \|\psi_x\|_{L^\infty(\rr)} \int_0^t \|\tilde{u}_{\la}^{\mu}(s)\|_{L^q(\rr)}^{q} ds \\
& \leq  C(\|\varphi_\la\|_{L^1(\rr)}, q) \|\psi_x\|_{L^\infty(\rr)} \int_0^t s^{-1+1/q} ds\\
& = C(\|\varphi\|_{L^1(\rr)}, q) \|\psi_x\|_{L^\infty(\rr)}  t^{1/q}.
\end{align*}

Using the fact that $\tilde{u}_{\la}^{\mu}$ has the same mass as the initial data we get
\begin{align*}
\mu \int_0^t  \int_{\rr} |\tilde{u}_{\la}^{\mu} \psi_{xx}(x) |dx ds & \leq  \mu  \|\psi_{xx}\|_{L^\infty(\rr)} \int_0^t \int_{\rr} \tilde{u}_{\la}^{\mu}(s, y) dy ds \\
&  =   \mu \|\psi_{xx}\|_{L^\infty(\rr)} \|\varphi\|_{L^1(\rr)}  t.
\end{align*}
Combining the estimates  above and using that $\mu\leq 1\leq \lambda$  it follows from \eqref{est.6} that
\begin{align*}
\int _{\rr}\tilde{u}_{\la}^{\mu}(t,x) \psi(x)dx\leq & \int _{\rr}|\varphi_\la(x)|\psi(x)dx\\
&+C(\|\varphi\|_{L^1(\rr)},J,q)  \left(   t \|\psi_{xx}\|_{L^\infty(\rr)}  +  t^{1/q}\|\psi_{x}\|_{L^\infty(\rr)} \right).
\end{align*}

In particular, let us choose as test function $\psi_R(x)=\psi(x/R)$ where $\psi\in C^\infty(\rr)$ is a fixed function that satisfies $0\leq \psi \leq 1$ such that  $\psi(x)\equiv 0$ for $|x|<1$ and $\psi(x)\equiv 1$ for $|x|>2$. We apply the above estimate with $\psi_R$.  Using the properties of the support of $\psi$  we get
\begin{align*}
\int_{|x|>2R} \tilde{u}_{\la}^{\mu} (t, x) dx & \leq  \int_{\rr} \tilde{u}_{\la}^{\mu}(t, x) \psi_R(x) dx \\
  &\leq  \int_{\rr} |\varphi_\la (x)| \psi_R (x)dx +C\l( \frac{t}{R^2} \|\psi_{xx}\|_{L^\infty(\rr)}+ \frac{t^{1/q}}{R} \|\psi_x\|_{L^\infty(\rr)}\r)\\
& \leq  \int_{|x|\geq R} |\varphi_\la(x)| dx +C_1(\psi, \|\varphi\|_{L^1(\rr)}, q, J)\l(\frac{t}{R^2} + \frac{t^{1/q}}{R}\r) \\
& =\int_{|x|> \la R} |\varphi(x)| dx + C_1(\psi, \|\varphi\|_{L^1(\rr)}, q, J)\l(\frac{t}{R^2} + \frac{t^{1/q}}{R}\r).
\end{align*}
Since $\la\geq 1$ the last estimate shows that \eqref{int.2r} holds for the nonnegative solution $\tilde{u}_{\la}^{\mu}$ and in consequence it also holds for the changing sign solution $u_{\la}^{\mu}$. The proof is now complete.
\end{proof}

\subsection{Compactness estimates}\label{changesign}
We will now give  very useful estimates in Lemmas \ref{comp.changing.sign} and \ref{comp.changing.sign2} below. The results of these lemmas are in the spirit of the results in \cite{1052.35126}
and \cite{1070.35065}. We emphasize that the
estimates in Lemma \ref{comp.changing.sign} are valid for any $\varphi\in L^1(\rr)$,  but they are not uniform with respect to $\lambda$. On the other hand the estimates in Lemma \ref{comp.changing.sign2} are uniform with respect to the parameter $\lambda$,  but they require nonnegative solutions, so nonnegative initial data $\varphi$.


\begin{lemma}\label{comp.changing.sign}
Let $\varphi\in L^1(\rr)$. There exists a nondecreasing function $\omega_\la=\omega_{\lambda,\varphi}:[0,\infty)\rightarrow [0,\infty)$ satisfying $\omega_\lambda(r)\rightarrow 0$ as $r\rightarrow 0$ such that   the solution of \eqref{eq.u.la} satisfies for all $t>0$ the following estimates
 \begin{equation}\label{shift.space}
 \int_{-\infty}^{\infty} |\umu_\la(t, x+h)-\umu_{\la}(t, x)|dx \leq \omega_\la(|h|), \quad \forall h\in \rr,
 \end{equation}
  \begin{multline}\label{shift.time}
  \int_{-\infty}^{\infty} |\umu_\la(t+k, x)-\umu_{\la}(t, x)|dx \\
  \leq C( q, J, \|\varphi\|_{L^1(\rr)})
 ( (\lambda^{q-2}+\mu)k^{1/3}+ k^{2/3}t^{-1+1/q}) + 4\omega_\la(k^{1/3}),
  \end{multline}
  for any $\la> 0$ and $\mu>0$, where $C(q, J, \|\varphi\|_{L^1(\rr)})$ is a positive constant depending on $q$, $J$ and $\varphi$.
 \end{lemma}

\begin{remark}
In contrast with the results in  \cite[Lemmas 2.3  and  3.12]{1052.35126}, here the constant $C( q, J, \|\varphi\|_{L^1(\rr)})$ in \eqref{shift.time} does not depend on the $L^\infty(\rr)-$ norm of the initial data $\varphi_\la$. This comes from the fact that in the proof we use the  $L^1(\rr)-L^\infty(\rr)$ estimates obtained  in Lemma  \ref{est.sol.negativas} for the solutions of \eqref{eq.u.la} instead of using that $\|\umu_{\la}(t)\|_{L^\infty(\rr)}\leq \|\varphi_\la\|_{L^\infty(\rr)}$.
\end{remark}

\begin{proof}
We recall  that for any function $\varphi\in L^1(\rr)$ there exists a nondecreasing function $\omega=\omega_\varphi: [0, \infty)\to [0, \infty)$, satisfying $\omega(r)\to 0$, as $r\to 0$,  the so-called $L^1$ modulus of continuity of $\varphi$,  such that
\[
  \int _{-\infty}^\infty |\varphi(x+h)-\varphi(x) |dx\leq \omega(|h|), \quad \forall \ h\in \rr.
\]
Let us observe that
\[
   \int _{-\infty}^\infty |\varphi_\la(x+h)-\varphi_\la(x) |dx= \int _{-\infty}^\infty |\varphi(x+\la h)-\varphi(x) |dx\leq \omega(\la |h|):=\omega_\lambda(|h|).
   \]

 For the justification of \eqref{shift.space} we apply the $L^1-$ contraction property established in Proposition \ref{max.principle} to get:
\begin{equation}\label{bounds.initial.data}
\|\umu_\la(t, \cdot)- \overline{u}_{\la}^{\mu}(t, \cdot)\|_{L^1(\rr)}\leq \|\varphi_\la-\overline{\varphi}_\la\|_{L^1(\rr)}, \quad \forall t>0,
\end{equation}
where $\overline{u}_{\la}^{\mu}$ is the solution of \eqref{eq.u.la} corresponding to some arbitrary initial datum $\overline{\varphi}_\la$. Taking $\overline{\varphi}_\la(x)=\varphi_\la(x+h)$ we obtain $\overline{u}_{\la}^{\mu}(t, x)=\umu_\la(t, x+h)$. Therefore, it follows that
$$\int_{-\infty}^{\infty} |\umu_\la(t, x+h)-\umu_\la(t, x)|dx \leq \int_{-\infty}^{\infty} |\varphi_\la(x+h)-\varphi_\la(x)|dx \leq \omega_\la(|h|). $$

  For the last statement of lemma, we proceed as in \cite[Lemma 2.3]{1052.35126},  but the estimate we derive takes into account the $L^1 \to L^\infty$ decay of solutions. Let $\phi$ be a smooth and bounded function which will be precise later. Then we multiply the equation in \eqref{eq.u.la} by $\phi$  and integrate by parts. For any fixed $k>0$ we obtain
\begin{align}\label{ident}
\int_{-\infty}^{\infty} & \phi(x) \left[(\umu_{\la}(t+k, x)-\umu_\la(t, x)\right] ds dx\nonumber \\
& = \la^q  \int_{-\infty}^{\infty} \int_{t}^{t+k} (J_\la \ast \umu_\la -\umu_\la )\phi ds dx +\frac{1}{q}\int_{-\infty}^{\infty} \int_{t}^{t+k} |\umu_\la|^{q-1}\umu_\la \phi_x ds dx\nonumber\\
& \quad + \mu  \int_{\rr}\int_{t}^{t+k} \umu_{\la}\phi_{xx}  ds  dx\nonumber\\
&:=I_1+I_2+\mu  I_3.
\end{align}

In the light of the computations in  \cite{1052.35126} we choose
$$\phi(x)= \int_{-\infty}^{\infty} k^{-1/3}\rho\left(\frac{x-\xi}{k^{1/3}}\right) \sgn\left(\umu_\la(t+k, \xi)-\umu_\la(t, \xi)\right)d\xi,$$
where $\rho$ is a nonnegative mollifier, supported in $[-1, 1]$  with mass one. It is easy to notice that
\begin{equation}\label{bounds.deriv}
|\phi|\leq 1, \quad |\phi_x|\leq C_1 k^{-1/3}, \quad |\phi_{xx}|\leq C_1 k^{-2/3},
\end{equation}
for some positive constant $C_1$.

Now, let us estimate the terms on the right hand side of \eqref{ident}.
We first estimate $I_1$.
Since $J_\la$ is an even function we have
$$I_1=\la^q \int_{t}^{t+k}  \int_{-\infty}^{\infty} \umu_\la(s, x) (J_\la \ast \phi -\phi)(x)  dx ds. $$
Using  estimate \eqref{bound.second.order} and
 the upper bound for the second derivative of $\phi$ in \eqref{bounds.deriv} we obtain
\[
  |I_1|\leq C(J) C_1 \la^{q-2} k^{-2/3}\int_{t}^{t+k} \int_{\rr} |\umu_\la(s, x)| dx ds.
\]
 Applying Proposition \ref{max.principle} the $L^1(\rr)$-norm of $\umu_\lambda$ does not increase, so we get
\begin{equation}
\label{first.term}
  |I_1| \leq C(J)C_1 \la^{q-2} k^{-2/3}\int_{t}^{t+k} \|\varphi_\la\|_{L^1(\rr)} ds
\leq  C(J)C_1 \la^{q-2} k^{1/3} \|\varphi\|_{L^1(\rr)}.
\end{equation}

In the case of $I_2$, from the upper bound of the first derivative of $\phi$ in \eqref{bounds.deriv} we obtain
\[
  |I_2|\leq \frac{1}{q} \int_{t}^{t+k} \int_{-\infty}^{\infty} |\umu_\la|^{q} |\phi_x| dx ds \leq C_1 k^{-1/3} \frac{1}{q} \int_{t}^{t+k} \|\umu_\la(s)\|_{L^q(\rr)}^{q} ds.
\]
Using Lemma \ref{est.sol.negativas}  for some positive constant $\tilde{C}(\|\varphi\|_{L^1(\rr)}, q)$ we have
\begin{equation}
  \label{ineq2}
|{I_2}| \leq \tilde{C}(\|\varphi\|_{L^1(\rr)}, q) k^{-1/3} \int_{t}^{t+k} s^{-1+1/q} ds \leq {C}(\|\varphi\|_{L^1(\rr)}, q) k^{2/3} t^{-1+1/q}.
\end{equation}

When considering $I_3$  it follows that
\begin{align}\label{ineq2.1}
|I_3| \leq \int_{\rr} \int_{t}^{t+k} &  |\umu_{\la}| |\phi_{xx}| ds dx  \leq C_1 k^{-2/3}  \int_{\rr} \int_{t}^{t+k} |\umu_\la| ds dx  \nonumber\\
&\leq C_1 k^{-2/3} \int_{t}^{t+k} \|\varphi_\la\|_{L^1(\rr)} ds =   C_1 k^{1/3} \|\varphi\|_{L^1(\rr)}.
\end{align}
From \eqref{first.term}, \eqref{ineq2} and \eqref{ineq2.1} we conclude that
\begin{equation}\label{equ1}
\int_{-\infty}^{\infty} \phi(x) |(\umu_{\la}(t+k, x)-\umu_\la(t, x)| ds dx \leq C( q, J, \|\varphi\|_{L^1(\rr)})
 ((\lambda^{q-2}+\mu)k^{1/3}+k^{2/3}t^{-1+1/q}).
\end{equation}
Now the proof  finishes in the same way  as in \cite[Lemma 2.3]{1052.35126}.
\end{proof}

\begin{lemma}\label{comp.changing.sign2}
Let $\varphi\in L^1(\rr)$ be nonnegative and $\umu_\lambda$ be the solution  of \eqref{eq.u.la}. There exists a smooth function $\omega=\omega_{\varphi, q}(t, r):(0,\infty)\times (0,1)\rightarrow (0,\infty)$ depending on $\varphi$ and $q$, which is nondecreasing in the second variable, such that, for any fixed $t>0$, $\omega(t, r)\rightarrow 0$ as $r\rightarrow 0$, and the following uniform estimates hold for any  parameters $0<\mu\leq 1\leq \lambda$:
 \begin{equation}\label{shift.space2}
 \int_{-\infty}^{\infty} |\umu_\la(t, x+h)-\umu_{\la}(t, x)|dx \leq \omega(t, |h|), \quad\quad  \forall\, |h|<1, \  \forall t>0.
 \end{equation}
 In addition, there exists a positive constant $C=C(\|\varphi\|_{L^1(\rr)}, q, J)$ such that
  \begin{multline}\label{shift.time2}
  \int_{-\infty}^{\infty} |\umu_\la(t+k, x)-\umu_{\la}(t, x)|dx
  \leq C \l((\la^{q-2}+\mu)k^{1/3}+ k^{2/3}t^{1/q-1}\r)\\
  + 2\omega(t+k, k^{1/3})+2\omega(t, k^{1/3}), \quad\quad  \forall 0<k<1, \ \forall t>0.
  \end{multline}
 \end{lemma}

\begin{remark}Function $\omega$ above can be written explicitly as
\[
    \omega_{\varphi,q}(t,r)=2 \int _{|x|>r^{-1/2}}\varphi(x) dx+ C(\|\varphi\|_{L^1(\rr)},q, J)r^{1/2} (t^{-2/q}+t^{-1/q}+t+t^{1/q}).
\]
In contrast with the results in  \cite[Lemmas 2.3  and 3.12]{1052.35126}, here the constant $C$ in \eqref{shift.time2} does not depend on the $L^\infty(\rr)$ norm of the initial data $\varphi_\la$. Moreover the estimates are uniform with respect to the parameters $0< \mu \leq 1\leq \la$.
\end{remark}

\begin{proof}
In the following the constant $C$ may change from line to line. First, we claim that there exists a positive constant $C=C(\|\varphi\|_{L^1(\rr)}, q)$ such that, for any $R>0$ and $h\in (0, 1)$ it holds
\begin{equation}\label{claim}
\int_{-R}^{R} |\umu_\la(t, x+h)-\umu_\la(t, x)|dx \leq C \l( \frac{Rh}{t^{2/q}}+\frac{h}{t^{1/q}}\r).
\end{equation}
Indeed, proceeding as in \cite[Corollary 4]{1070.35065} we define the set
\[
  \mathcal{P}:=\{x\in (-R, R), \  \umu_\la(t, x+h)\geq \umu_\la(t, x)\}
\]
 and we obtain
\begin{align*}
\int_{-R}^{R}& |\umu_\la(t, x+h)-\umu_\la(t, x)|dx \\
&\leq 2 \int_{\mathcal{P}} (\umu_\la(t, x+h)-\umu_\la(t, x))+\int_{-R}^{-R+h} \umu_\la(t, x) dx - \int_{R}^{R+h} \umu_\la(t, x) dx.
\end{align*}
Using  the mean value theorem and estimates \eqref{estt.6} and \eqref{est.infty} for the nonnegative solution $\umu_\la$
we successively obtain
\begin{align*}
\int_{-R}^{R} |\umu_\la(t, x+h)-\umu_\la(t, x)|dx&\leq 2 C(\|\varphi\|_{L^1(\rr)}, q) \int_{\mathcal{P}} \frac {h}{t^{2/q}} dx + 2h \|\umu_\la(t)\|_{L^\infty(\rr)}\\
& \leq C(\|\varphi\|_{L^1(\rr)}, q) \l( \frac{Rh}{t^{2/q}} +  \frac{h}{t^{1/q}}\r).
\end{align*}

Next, we set $I$ to be the left hand side term in \eqref{shift.space2} and split it as  $I=I_1+I_2$
where
\[
 I_2= \int_{|x|\geq 3 R} |\umu_{\la}(t, x+h)-\umu_{\la}( t, x)|dx,
\]
 with a large constant $R$ which will be precise later. First, the term $I_1$ can be estimated  using \eqref{claim}. In the case of $I_2$, when $R>h$ we have
\[
  I_2\leq 2 \int _{|x|\geq 3R-h} u_\lambda^\mu(t,x)dx\leq 2 \int _{|x|>2R}u_\lambda^\mu(t,x)dx.
\]
Using the tail control obtained in Lemma \ref{tails.control} we get
\[
  I_2\leq 2 \int _{|x|>R}\varphi(x) dx + C(J,\|\varphi\|_{L^1(\rr)},q) \left(\frac t{R^2}+\frac{t^{1/q}}R\right).
\]
Hence our left hand side term in \eqref{shift.space2} satisfies
\[
  I\leq 2 \int _{|x|>R}\varphi(x) dx + C \left(\frac{Rh}{t^{2/q}} +  \frac{h}{t^{1/q}}+\frac t{R^2}+\frac{t^{1/q}}R\right).
\]
Choosing now $R=h^{-1/2}$ and using that $h<1$ we get
\[
  I\leq 2 \int _{|x|>h^{-1/2}}\varphi(x) dx+ Ch^{1/2} (t^{-2/q}+t^{-1/q}+t+t^{1/q}).
\]
Denoting
\[
  \omega(t,r)=2 \int _{|x|>r^{-1/2}}\varphi(x) dx+ Cr^{1/2} (t^{-2/q}+t^{-1/q}+t+t^{1/q})
\]
we obtain the desired estimate.

  For the last statement of lemma, we can proceed as in \cite[Lemma 2.3]{1052.35126}.
  \end{proof}
  Now we have all the ingredients to prove Theorem \ref{global.convergence}.

\begin{proof}[Proof of Theorem \ref{global.convergence}]
Let us assume without loosing generality that $\alpha=1$ in system \eqref{claws}.
We consider the regularized system \eqref{eq.u.la} with $\lambda=1$ and denote its solution by $u^\mu$.
In view of Lemmas \ref{tails.control} and \ref{comp.changing.sign}  we obtain that $(u^\mu)_{\mu}$ is relatively compact in $C([0,T],L^1(\rr))$ for any $0<T<\infty$. As a consequence there is a function $u \in C([0,\infty),L^1(\rr))$ such that, up to a subsequence,
$u^\mu \rightarrow u$ in $C([0,T],L^1(\rr))$. Moreover
\[
  \int _{\rr} u(t,x)dx=\int _{\rr}\varphi(x)dx.
\]

Following the ideas in \cite{1052.35126} we obtain that $u\in C([0,\infty),L^1(\rr))\cap L^\infty_{loc}((0,T),L^\infty(\rr))$ is the entropy solution of problem \eqref{claws.2}.

 Now for any positive time $t$ we have that $\umu(t)\rightarrow u(t)$ a.e. and so the $L^\infty(\rr)$-bound \eqref{est.neg.1} on $\umu$ will be transferred to $u(t)$:
	\begin{equation}\label{mod}
  0\leq |u(t,x)|\leq \Big(\frac {qM}{(q-1)t}\Big)^{1/q}.
\end{equation}
This implies that the convergence of $\umu(t)$ towards $u(t)$ will also hold in any $L^p(\rr)$ with $1\leq p<\infty$. Since the uniqueness of the entropy solution in the norm $L^\infty([0,\infty),L^1(\rr))\cap L^\infty_{loc}((0,\infty),L^\infty(\rr))$  was proved in \cite[Theorem 2, p. 497]{MR2103702} the whole sequence $\umu$ converges to $u$.
\end{proof}

\begin{proof}[Proof of Theorem \ref{oleinik_theorem}]
The two estimates \eqref{ineq-oleinik}-\eqref{est.u.5} are consequence of the fact that $\umu$ solution of the regularized system \eqref{eq.u.la} (with $\lambda=1$) strongly converges to $u$ in $L^p(\rr)$ for $1\leq p<\infty$. When $p=\infty$ we use the a.e. convergence to obtain the desired result.
\end{proof}

\section{Asymptotic behaviour of nonnegative solutions}\label{asybehavior}

In this section we prove Theorem \ref{asimp} in the case of nonnegative solutions.
We first obtain the compactness of the trajectories $\{u_\la\}_{\la\geq 1}$ and prove their convergence to an entropy solution of problem \eqref{gas}. Finally, the results of Theorem \ref{asimp} are obtained by showing that $u_\la(1)\to w_M(1)$, where $w_M$ is the unique entropy solution of system \eqref{gas}.
We proceed in several steps as follows.
\vspace{0.2cm}

\noindent{\bf Step I. Compactness.} As in the proof of Theorem \ref{global.convergence} for any fixed $\la>0$ the family $(u_\la^\mu)_{\mu}$ of the regularized system \eqref{eq.u.la} is relatively compact in $C([0, \infty), L^1(\rr))$ and there exists the limit function $u_\la$ such that
$$u_\la^\mu \to u_\la, \textrm{ in } C([0, \infty), L^1(\rr)), \quad \textrm{ as } \mu\to 0.$$
  Then for any positive time $t>0$,   since $\umu_\lambda (t) \rightarrow u_\la(t)$ in $L^1(\rr)$ as $\mu\rightarrow 0$, we first observe that  we can pass to the limit in estimates \eqref{shift.space2} and \eqref{shift.time2}. Let us fix $0<t_1<t_2<\infty$. We  obtain that  for some function $\omega$, independent on $\la\geq 1$,  such that $\omega(r)\rightarrow 0$ as $r\rightarrow 0$
the following estimates  holds uniformly for any $\la\geq 1$:
\begin{equation}
\label{est.comp.1}
  \max _{t\in [t_1,t_2]} \int_{\rr} |u_\la (t,x+h)-u_\la (t,x)|dx\leq C(t_1,t_2)\omega(h)
\end{equation}
and
\begin{equation}
\label{est.comp.2}
\max _{t\in [t_1,t_2]} \int_{\rr} |u_\lambda(t+k,x)-u_\lambda (t,x)|dx\leq C(t_1,t_2)\omega (h).
\end{equation}
Moreover we have mass conservation and the $L^\infty(\rr)$-bound on $u_\la$:
\begin{equation}
\label{mass.conservation}
  \int _{\rr}u_\lambda(t,x)dx=M, \quad \|u_\lambda(t)\|_{L^\infty(\rr)}\leq Ct^{-1/q}, \quad \forall t>0.
\end{equation}
The tail control obtained in Lemma \ref{tails.control} also transfers from $\{\umu_{\la}\}$ to the family $\{u_\lambda\}$:
\begin{equation}
\label{tail.u.lambda}
  \int _{|x|>2R} u_\lambda(t,x)dx\leq \int _{\rr}{\varphi}(x)dx + C\left(\frac t{R^2}+\frac{t^{1/q}}R\right).
\end{equation}
Classical compactness arguments give us that the family $\{u_\lambda\}_{\lambda>1}$ is relatively compact in $C([t_1,t_2],L^1(\rr))$ for any $0<t_1<t_2$. Then there exists a function $\overline{u}\in C((0,\infty),L^1(\rr))$ such that, up to a subsequence, $u_\lambda \rightarrow \overline{u}$ in
 $C([t_1,t_2],L^1(\rr))$ for any $0<t_1<t_2$. The limit point $\overline{u}$ also satisfies
 \begin{equation}
\label{prop.u}
  \int_\rr \overline{u}(t,x)dx=M, \quad \|\overline{u}(t)\|_{L^\infty(\rr)}\leq Ct^{-1/q}.
\end{equation}
Moreover, the above convergence also holds in $C([t_1,t_2],L^p(\rr))$ for any $0<t_1<t_2$ and $1\leq p<\infty$.

\vspace{0.2cm}

\noindent{\bf Step II. Identification of the limit.}\label{identlimit} Due to the compactness results above, we can show that $\overline{u}$ is the entropy solution of problem
  \begin{equation}\label{limitt}
\left\{
\begin{array}{ll}
w_t  + ({|w|^{q-1}w}/q)_x= 0,& x\in \rr, t>0,\\[10pt]
w(x,0)=M\delta_0, &x\in \rr.
\end{array}
\right.
\end{equation}
We now check that the cluster limit point $\overline{u}$ of $u_\la$ satisfies the Kru\v{z}kov inequality C3) in Definition \ref{entropysol}. Since $u_\la$ is an entropy solution of system \eqref{claws.2} it satisfies
\begin{multline*}
\int_0^\infty \int_{\rr} \left(|u_\la-k|\frac{\partial \phi}{\partial t} +\sgn(u_\la-k)(f(u_\la)-f(k))\frac{\partial \phi}{\partial x} \right) dxdt\\
\geq  \la^q\int_0^\infty \int _{\rr}\left(|u_\la-k|-\sgn(u_\la-k)J_\la\ast (u_\la-k)\right)\phi dxdt,
\end{multline*}
for any $\phi \in C_c^\infty((0, \infty) \times \rr)$, $\phi\geq 0$.
 Following the ideas of \cite[Lemma 4.3]{MR2914243} we will prove that letting  $\lambda
\rightarrow \infty$ the term on the right hand side vanishes. Observe that
\begin{multline*}
(|u_\la(t, \cdot)-k|-\sgn(u_\la(t, \cdot)-k)J_\la\ast (u_\la(t, \cdot)-k))(x)\\
=
\int _{\rr}J_{\la}(x-y)\left(|u_\la(t, x)-k| - \sgn (u_\la(t, x)-k)(u_\la(t, y)-k)\right) dy\\
\geq
\int _{\rr}J_{\la}(x-y)\left(|u_\la(t, x)-k| -|u_\la(t, y)-k|\right) dy.
\end{multline*}
Employing a change of variables for $\phi\geq 0$ we get
\begin{multline*}
 \int _{\rr}\left(|u_\la(t, x)-k|-\sgn(u_\la(t, x)-k)J_\la\ast (u_\la(t, \cdot)-k)(x)\right) \phi(t, x) dx\\
\geq \int _{\rr}\int _{\rr} |u_\la(t, x)-k|  J_{\la}(x-y)(\phi(t,  x)-\phi(t, y))dydx\\
=\int _{\rr} |u_\la(t, x)-k|  (\phi(t, x)-J_\la\ast \phi(t, x))dx.
\end{multline*}

Let us assume that $\phi$ is supported in $[0,T]\times \rr$. Then we have
\begin{align*}
 \la ^q
\int _0^T \int _{\rr}& |u_\la(t, x)-k|  |\phi(t,x)-J_\la\ast \phi(t,x)|dxdt\\
&\leq
\la^{q-2} \|\la^2(J_\lambda\ast \phi -\phi)\|_{L^\infty((0,T)\times \rr)} \int _0^T\int_{\rr} |u_\la(t, x)|  dxdt\\
&\quad+ \lambda^{q-2}|k| \int _0^T  \|\la^2(J_\lambda\ast \phi -\phi)\|_{L^1(\rr)}dt.
\end{align*}
  Using estimate \eqref{bound.second.order} and since  $u_\la$ is uniformly bounded in $L^1(\rr)$
we get
\begin{align*}
 \la ^q
\int _0^T \int _{\rr}& |u_\la(t, x)-k|  |\phi(t,x)-J_\la\ast \phi(t,x)|dxdt\\
&\leq
\la^{q-2} T C(J)\|\phi_{xx}\|_{L^\infty((0,T)\times \rr)}  \|\varphi\|_{L^1(\rr)} + \lambda^{q-2}|k| T C(J) \|\phi_{xx}\|_{L^\infty((0,T),L^1(\rr))} .
\end{align*}
Using that
 $q<2$ we obtain that the right hand side term goes to zero as $\la\rightarrow \infty$.

Using now that $u_\la$ strongly converges to $\overline{u}$ in $L^1(\rr)$ we can pass to the limit also in the left hand side of the Kru\v{z}kov inequality. This gives us
that the cluster point $\overline{u}$ satisfies the inequality in condition C3):
\[
\int_0^\infty \int _{\rr} \left(|\overline{u}-k|\frac{\partial \phi}{\partial t} +\sgn(\overline{u}-k)(f(\overline{u})-f(k))\frac{\partial \phi}{\partial x} \right) dxdt\geq 0.
\]

Analogously, we can obtain the passage to the  limit in the weak formulation of equation \eqref{claws}.

\vspace{0.2cm}

\noindent{\bf Step III. Initial data.}\label{intialdata}
We have to prove that $\overline{u}$ takes initial data $M\delta _0$ in the sense of measures. We use the same arguments as in case of the tails control.
We recall that we have obtained in the proof of Lemma \ref{tails.control} that for any  $\psi\in C^2(\rr)$ there exists a uniform constant $C>0$ such that the following holds for any $\lambda\geq 1$:
\[
\Big|\int _{\rr}u_\la(t,x) \psi(x)dx-\int _{\rr}\varphi_\la(x)\psi(x)dx \Big|
\leq C \left( t \|\psi_{xx}\|_{L^\infty(\rr)}  +  \|\psi_{x}\|_{L^\infty(\rr)} t^{1/q}\right).
\]
Letting $\lambda \rightarrow\infty$ we get
for any $\psi\in C^2(\rr)$
\[
\Big|\int _{\rr}\overline{u}(t,x) \psi(x)dx-M\psi(0) \Big|
\leq C \left( t \|\psi_{xx}\|_{L^\infty(\rr)}  +  \|\psi_{x}\|_{L^\infty(\rr)} t^{1/q}\right).
\]
By a density argument we obtain that for any $\psi\in BC(\rr)$ we have
\[
\lim _{t\rightarrow 0} \int _{\rr}\overline{u}(t,x) \psi(x)dx=M\psi(0).
\]
This shows that $\overline{u}$ is an entropy solution of problem \eqref{limitt}. Since by \cite{MR735207} we have a unique entropy solution,  the whole sequence $\{u_\la\}_{\la}$ converges toward $\overline{u}=w_M$ given by \eqref{limitt}.

In the case of bounded continuous functions $\psi$ we have to use an approximation argument and the control of the tails given by Lemma \ref{tails.control}.

\vspace{0.2cm}

\noindent{\bf Step IV. Asymptotic behavior.} Roughly speaking, the asymptotic behavior is equivalent with the strong convergence in $L^1$ of the family $\{u_\la\}$ to $w_M$.
From Steps I-III we deduce that the  nonnegative solutions satisfy
\begin{equation}\label{mod2}
\|u_\lambda (1)-w_M(1)\|_{L^1(\mathbb{R})}\rightarrow 0,\quad \text{as} \quad\lambda \rightarrow \infty.
\end{equation}

This proves the result of Theorem \ref{asimp} when $p=1$.

According to \cite{MR1233647} it is well-known that solutions of the pure convective equation \eqref{limitt} in the exponent range $1< q <2$, with Dirac initial data, behave asymptotically for large time as the solution of the heat equation with the corresponding convective term, namely,
\begin{equation}\label{decay.profile}
\|w_M(t)\|_{L^\infty(\mathbb{R})}\lesssim t^{-\frac 1q},
\end{equation}
which is the same decay as for the $\overline{u}$ in \eqref{prop.u}. If $u$ is the entropy solution of system \eqref{Ldifusion}  then, using the interpolation inequality it follows that
\begin{align*}
\|u(t)-w_M(t)\|_{L^p(\mathbb{R})}&\lesssim \|u(t)-w_M(t)\|_{L^1(\mathbb{R})}^{1/p} (\|u(t)\|_{L^\infty(\rr)}^{(p-1)/p}+\|w_M(t)\|_{L^\infty(\rr)}^{(p-1)/p}),
\end{align*}
and we obtain the desired estimate according to \eqref{mod}, \eqref{mod2} and \eqref{decay.profile}.
The proof of the main result is now finished.

\section{Changing sign solutions
}\label{signchange}

In this section we prove Theorem \ref{asimp} in the case of changing sign solutions. In this case Oleinik-type estimates cannot be obtained  and the tools employed in the previous section do not work here. We apply a technique based on the classical flux-entropy method,  but taking care on the behavior of the nonlocal terms.
%
%
%
 Let us consider $\umu_\lambda$ solution of the regularized system
\begin{equation}\label{reg.system}
\left\{
\begin{array}{ll}
\umu_{\la,t}+ |\umu_\la|^{q-1}(\umu_\la)_x=\la^{q} (J_\la\ast \umu_\la -\umu_\la) +\mu (\umu_\la)_{xx},& x\in \rr,\ t>0,\\[10pt]
\umu_\la(0)=\varphi_\la.
\end{array}
\right.
\end{equation}

We consider a convex function  $\Phi\in C^2(\rr)$ and consider $\Psi$ such that $\Psi '(y)=\Phi'(y)f'(y)$ for all $y\in \rr$. This is always possible since $f(u)=|u|^{q-1}u$, $q>1$, is $C^1(\rr)$. Multiplying \eqref{reg.system} by $\Phi'(\umu_\lambda)$ we obtain
\[
  [\Phi(\umu_\lambda)]_t+[\Psi (\umu_\lambda)]_x=\lambda^q (J_\lambda\ast \umu_\lambda -\umu_\lambda)\Phi'(\umu_\lambda)
  +\mu [(\Phi(\umu_\lambda))_{xx}-\Phi''(\umu_\lambda)(\umu_{\lambda,x})^2].
\]
Since $\umu_\lambda \rightarrow u_\lambda$ in $C([0,T], L^1(\rr))$ and $(\umu_\lambda)_{\mu>0}$ is uniformly bounded on any interval $I\subset (0,\infty)$ we obtain that
\begin{equation}
\label{flux.lambda}
   [\Phi(u_\lambda)]_t+[\Psi (u_\lambda)]_x=\lambda^q (J_\lambda\ast u_\lambda -u_\lambda)\Phi'(u_\lambda)\quad \text{in}\quad \mathcal{D}'((0,\infty)\times \rr).
\end{equation}

We are going to prove that, for any $0<t_1<t_2<\infty$, the right hand side in \eqref{flux.lambda} is the sum between a compact set in
$H^{-1}_{loc}(((t_1,t_2)\times \rr)$ and a bounded set in the space of Radon Measures $\mathcal{M}((t_1,t_2)\times \rr)$. Before proving that let us recall that inequality \eqref{gradient.est} transfer to $u_\lambda$: for any $0<t_1<t_2<\infty$ the following inequality holds uniformly in $\lambda$:
\begin{equation}\label{gradient.est.u.lambda}
  \lambda^q\int_{t_1}^{t_2}\int_\rr \int _\rr J_\lambda(x-y)(u_\lambda(t,x)-u_\lambda(t,y))^2dxdydt\leq C(t_1,\|\varphi\|_{L^1(\rr)}).
\end{equation}
Also inequality \eqref{est.neg.1} transfers to $u_\lambda$:
\begin{equation}
\label{infty.u.lambda}
  \|u_\lambda\|_{L^\infty((t_1,t_2)\times \rr)}\leq C(t_1, \|\varphi\|_{L^1(\rr)}).
\end{equation}

Let us denote by $I_\lambda=\lambda^q (J_\lambda\ast u_\lambda -u_\lambda)\Phi'(u_\lambda)$ and by $<\cdot,\cdot>$ the scalar product on $L^2((t_1,t_2)\times\rr)$. Using a changes of variables we get
\begin{align*}
\label{}
  <I_\lambda,\phi>&=\lambda^q \int_{t_1}^{t_2}\int_\rr (J_\lambda\ast u_\lambda -u_\lambda)\Phi'(u_\lambda)\phi  dxdt\\
  &=
  -\frac{\lambda^q}2  \int_{t_1}^{t_2}\int_\rr\int_\rr J_\lambda(x-y)(u_\lambda(t,x)-u_\lambda(t,y))[  (\Phi'(u_\lambda)\phi)(t,x)-
  (\Phi'(u_\lambda)\phi)(t,y)] dx dy dt\\
  &=  <I_\lambda^1,\phi>+  <I_\lambda^2,\phi>,
\end{align*}
 where
\[
   <I_\lambda^1,\phi>= -\frac{\lambda^q}2  \int_{t_1}^{t_2}\int_\rr\int_\rr J_\lambda(x-y)(u_\lambda(t,x)-u_\lambda(t,y))[  (\Phi'(u_\lambda))(x)-
  (\Phi'(u_\lambda))(y)]\phi(t,x) dx dy dt
\]
and
\[
  <I_\lambda^2,\phi>= -\frac{\lambda^q}2  \int_{t_1}^{t_2}\int_\rr\int_\rr J_\lambda(x-y)(u_\lambda(t,x)-u_\lambda(t,y)) (\phi(t,x)-\phi(t,y))    \Phi'(u_\lambda)(t,y) dx dy dt.
\]
Since $\Phi$ is $C^2$ and $|u_\lambda|$ is uniformly bounded, say by $M$,  we find that $\Phi''$ is also uniformly bounded on the interval
$[- M,M  ]$
and, due to \eqref{gradient.est.u.lambda},  the following holds
\begin{align}
\label{est.i.1}
  |  <I_\lambda^1,\phi>| &\lesssim  \lambda^q \|\Phi''\|_{L^\infty(-M,M)} \|\phi\|_{L^\infty((t_1,t_2)\times \rr)}  \nonumber\\
   &\quad  \times \int_{t_1}^{t_2}\int_\rr\int_\rr J_\lambda(x-y)(u_\lambda(t,x)-u_\lambda(t,y))^2 dx dy dt \\
\nonumber &\lesssim  C(t_1, \|\varphi\|_{L^1(\rr)}) \|\Phi''\|_{L^\infty(-M,M)} \|\phi\|_{L^\infty((t_1,t_2)\times \rr)}.
\end{align}
This shows that $I_\lambda^1$ is bounded in the space of Radon measures $\mathcal{M}((t_1,t_2)\times \rr))$.

We now evaluate $I_\lambda^2$. Let us recall that for any  $J\in L^1(1+|x|^2)$ the following inequality \cite[Lemma 2.3]{Ignat:2013fk}) holds
\begin{equation}
\label{ineqsa}
    \lambda^2\int _{\rr}\int _{\rr} J_\lambda(x-y)(\phi(x)-\phi(y))^2dxdy\leq C(J)\int _{\rr} u_x^2(x)dx.
\end{equation}
 Applying H\"{o}lder inequality, inequality \eqref{ineqsa} and \eqref{gradient.est.u.lambda} we get
\begin{align}
\label{est.i.2}
    |  <I_\lambda^2,\phi>|^2 &\lesssim \lambda^{q-2} \|\Phi'\|^2_{L^\infty(-M,M)}  \lambda^q \int_{t_1}^{t_2}\int_\rr\int_\rr J_\lambda(x-y)(u_\lambda(t,x)-u_\lambda(t,y))^2  dx dy dt \\
\nonumber    &\quad \times \lambda^2
     \int_{t_1}^{t_2}\int_\rr\int_\rr J_\lambda(x-y)  (\phi(t,x)-\phi(t,y))^2 dx dydt \\
 \nonumber    &\lesssim  \lambda^{q-2} C(t_1,\|\varphi\|_{L^1(\rr)})\|\Phi'\|^2_{L^\infty(-M,M)}   \|\phi\|_{L^2((t_1,t_2),H^1(\rr))}^2 .
\end{align}
Thus
\[
  \|I_\lambda^2\|_{L^2((t_1,t_2),H^{-1}(\rr))}\lesssim \lambda^{\frac{q-2}2}.
\]
This implies in particular that $I_\lambda^2\rightarrow 0$ in $H^{-1}((t_1,t_2)\times \rr)$ so $I_\lambda^2$ is relatively compact
in $H^{-1}((t_1,t_2)\times \rr)$.

We now apply the arguments in \cite{MR584398}. We recall that $(u_\lambda)_{\lambda>0}$ is uniformly bounded in $(t_1,t_2)\times \rr$. Thus  we have that, up to a subsequence,
	\[
  u_\lambda \mathop{\rightharpoonup }^{*} \overline{u}\quad \text{in}\quad L^\infty((t_1,t_2)\times \rr).
\]
Using  that the right-hand side of \eqref{flux.lambda} is the sum between a relatively compact set in $H^{-1}_{loc}((t_1,t_2)\times \rr)$ and a bounded set in $\mathcal{M}((t_1,t_2)\times \rr))$
 by classical results of Tartar (\cite[Th. 26, p. 202]{MR584398}) we obtain that $f'(u_\lambda)\rightarrow f'(\overline{u})$ in $L^p_{loc}((t_1,t_2)\times\rr)$ for any $1\leq p<\infty$. Since the function $f$ is nowhere affine we obtain that
$u_\lambda \rightarrow \overline{u}$ in $L^p_{loc}((t_1,t_2)\times\rr)$ for any $1\leq p<\infty$. A diagonal process guarantees that
in fact
\[
  u_\lambda \rightarrow \overline{u} \quad \text{in}\quad L^p_{loc}((0,\infty)\times\rr), \, 1\leq p<\infty.
\]

Proceeding as in the Step II of  the case of nonnegative solutions (see Section \ref{asybehavior}) we  find that $\overline{u}$ is an entropy solution of problem \eqref{limitt}. Moreover the tail control obtained in Lemma \ref{tails.control} guarantees that
\[
  \int _{|x|>2R}|u_\lambda(t,x)|dx\leq \int _{\rr}|\varphi(x)|dx+C\left(\frac{t}{R^2}+\frac{t^{1/q}}R\right).
\]
This shows that for every $0<\tau<T<\infty$
\[
  u_\lambda \rightarrow \overline{u} \quad \text{in} \quad L^1((\tau,T)\times\rr).
\]

In order to prove that $\overline{u}$ take the initial data $M\delta_0$ in the sense of bounded measures we can proceed as in  \cite[ Th. 1, Step III, p. 56]{MR1233647}. Since equation \eqref{limitt} has a unique solution we obtain that $\overline{u}=w_M$ and the whole sequence $u_\lambda$ converges to $w_M$
.

Using the same arguments as in \cite[ Th. 1, Step IV, p. 57]{MR1233647} we obtain that
the entropy solution $u$ of system \eqref{Ldifusion} satisfies
\[
\lim _{t\rightarrow \infty} t^{\frac  1q\left(1-\frac 1p\right)}\|u(t)-w_M(t)\|_{L^p(\rr)}=0, \quad 1\leq p < \infty,
\]
which finishes the proof of the main result of this paper in the case of changing sign solutions.

\textbf{Acknowledgments.} C. C. and L. I. were partially supported by  a grant of the Romanian National Authority for Scientific Research and Innovation, CNCS--UEFISCDI, project number PN-II-ID-PCE-2011-3-0075.  L. I. was also partially supported by  Grant MTM2014-52347, MICINN, Spain and FA9550-15-1-0027 of AFOSR.
 A. P. was partially supported by CNPq (Brasil). This work started when A. P.  visited the ``Simion Stoilow" Institute of Mathematics of the Romanian Academy in Bucharest to whom the authors are very grateful for its hospitality and good research atmosphere.

\bibliographystyle{plain}
\bibliography{biblioLast}

\end{document}